%% file: Article.tex
\documentclass[a4paper, 10pt]{article}

\usepackage{amsmath, amssymb, amsthm}
\usepackage[CJKbookmarks=true,colorlinks,linkcolor=black,anchorcolor=black,citecolor=black]{hyperref}
\usepackage{CJK}
\usepackage{multirow}
\usepackage{booktabs}
\usepackage{appendix}
\usepackage{tikz}
\usepackage{graphicx, xcolor}
\usepackage{overpic,subfig, float}
\usepackage{pgfplots}
\usepackage[stable]{footmisc}
\usepackage{listings, color}
\newenvironment{@abssec}[1]{%
       \vspace{.05in}\footnotesize
	   \noindent
         {\upshape\bfseries #1 }\ignorespaces
	 }

\newcommand\keywordsname{Keywords}
\newcommand\AMSname{AMS subject classifications}
\newcommand\Ackname{Acknowledgments}
\newcommand\datav{Data Availibility}
\newenvironment{keywords}{\begin{@abssec}{\keywordsname}}{\end{@abssec}}
\newenvironment{AMS}{\begin{@abssec}{\AMSname}}{\end{@abssec}}
\newenvironment{Ack}{\begin{@abssec}{\Ackname}}{\end{@abssec}}
\newenvironment{datava}{\begin{@abssec}{\datav}}{\end{@abssec}}

\newcommand\bbR{\mathbb{R}}

\newcommand\bbN{\mathbb{N}}

\def\+#1{\boldsymbol{#1}}
\newcommand\ang[1]{\left\langle {#1} \right\rangle}

\newcommand\pd[2]{\dfrac{\partial {#1}}{\partial {#2}}}

\newcommand\odd[1]{\dfrac{\mathrm{d}}{\mathrm{d} {#1}}}
\newcommand\od[2]{\dfrac{\mathrm{d} {#1}}{\mathrm{d} {#2}}}

\newcommand\ifup[1]{}

\newcommand\tl[1]{\boldsymbol{\tilde{#1}}}

\begin{document}
\begin{CJK*}{UTF8}{gkai}
\theoremstyle{plain}
\newtheorem{lemma}{Lemma}[section]
\newtheorem{theorem}{Theorem}[section]
\newtheorem{cor}{Cor}[section]
\newtheorem{prop}{Prop}[section]
\newtheorem{remark}{Remark}[section]

\theoremstyle{definition}
\newtheorem{defn}{Definition}[section]
\newtheorem{exmp}{Example}[section]
\numberwithin{equation}{section}

\bibliographystyle{plain}

\title{On Well-posed Boundary Conditions for the Linear
	Non-homogeneous Moment Equations in Half-space}

\author{Ruo Li\thanks{CAPT, LMAM \& School of Mathematical Sciences,
    Peking University, Beijing 100871, China, email: {\tt
      rli@math.pku.edu.cn}.} \and Yichen Yang\thanks{School of
    Mathematical Sciences, Peking University, Beijing 100871, China, email:
    {\tt yichenyang@pku.edu.cn}.}
}

\maketitle{}

\input{intro}

\input{back}

\input{draw}

\input{proof}

\input{conclu}

\begin{appendix}
	\bibliography{../ATC}
\end{appendix}

\end{CJK*}
\end{document}

%% file: intro.tex
\begin{abstract}	
	We propose a necessary and sufficient condition for the well-posedness
	of the linear non-homogeneous Grad moment equations in half-space.
	The Grad moment system is based on Hermite expansion and regarded
	as an efficient reduction model of the Boltzmann equation. At a 
	solid wall, the moment equations are commonly equipped with
	a Maxwell-type boundary condition named the Grad boundary
	condition. We point out that the Grad boundary condition is
	unstable for the non-homogeneous half-space problem. Thanks
	to the proposed criteria, we verify the well-posedness
	of a class of modified boundary conditions. The technique to
	make sure the existence and uniqueness mainly
	includes a well-designed preliminary simultaneous transformation
	of the coefficient matrices,
	and Kreiss' procedure about the linear boundary value problem with
	characteristic boundaries. The stability is established by a
	weighted estimate. At the same time, we obtain the analytical
	expressions of the solution, which may help solve the half-space
	problem efficiently.
\end{abstract}

\section{Introduction}
Half-space problems are at the center of our understanding of the
kinetic boundary layer \cite{Sone2007}. The relevant study helps 
prescribe slip boundary conditions for the fluid-dynamic-type 
equations \cite{Sone2007} and results in the interface coupling 
condition between kinetic and hydrodynamic 
equations \cite{1997Arnold,Lu2017,Chen2019}.
For the Boltzmann equation, the theory of linear 
half-space problems has been well-developed 
(cf. \cite{Bardos2006} and references therein).

In this paper, we focus on the linear steady non-homogeneous 
equations in half-space:
\begin{gather}
\label{eq:und}
\+A\od{W(y)}{y} = -\+QW(y) + h(y),\quad y\in[0,+\infty), \\
\label{eq:und_bc}
\+BW(0) = g,
\end{gather}
where $e^{ay}W(y)\in L^2(\bbR_+;L^2(\bbR^N))$ and 
$e^{ay}h(y)\in L^2(\bbR_+;L^2(\bbR^N))$ for some constant $a>0$. 
Here $\+A,\+Q,\+B$ are given $N\times N$ matrices, and $g$ is 
an $N\times 1$ vector, all with constant coefficients.
For the considered Grad moment equations, the matrix $\+A$ is
symmetric, and $\+Q$ is symmetric positive semi-definite,
which individually has a special block structure, as the article
will show later.

We can realize the importance of the half-space 
problem \eqref{eq:und} from two aspects. First, the 
moment equations proposed
by Grad \cite{Grad1949} are regarded as an efficient reduction 
model \cite{framework} of the Boltzmann equation. As an extension
of the celebrated Navier-Stokes equations, the moment equations 
have gained much more attention in recent years 
\cite{Torrilhon2009,Fan_new,Hu2019,Cai2020}.
The half-space problem \eqref{eq:und} could arise
from the asymptotic analysis of the moment equations. This
problem plays a crucial role in understanding the boundary
layer of the moment equations \cite{2008Linear,Lijun2017}.
Second, the moment system \eqref{eq:und} may serve as an
efficient numerical solver of kinetic equations.
There are many approximatively analytical and numerical methods
to solve half-space kinetic layer equations.
Some most famous analytical methods are essentially low-order
moment methods, including the Maxwell method \cite{CC1989},
Loyalka's method \cite{Loyalka1967,Loyalka1968}, and 
the half-range moment method \cite{Gross1957}.
The arbitrary order moment equations are also widely used
to resolve the layer problems \cite{2008Linear,Gu2014,Lijun2017},
which can give formal analytical solutions.

Despite the broad application range, the well-posedness
of the moment system with the Grad boundary condition \cite{Grad1949}
is doubtable. For the linear initial boundary value problem,
\cite{Sarna2018} has pointed out that the Grad boundary condition
is unstable. Their method can not apply to the half-space problem
directly. We have studied the half-space problem for the linear 
homogeneous moment equations and proposed several solvability
conditions in an earlier paper \cite{Yang2022a}. The object
of this work is to study the well-posed criteria of the linear 
non-homogeneous Grad moment equations in half-space.

In general, the concept of well-posedness contains three points.
First, the solution exists. Second, the solution is unique. Third,
the solution is controlled by the boundary data and non-homogeneous term. 
In the homogeneous case, there is no need to consider the
variants of the non-homogeneous item. Due to this significant
difference, the non-homogeneous case is not a trivial 
corollary of the homogeneous situation.

Our method is essentially basic linear algebra. 
Because the matrices $\+A$ and $\+Q$ may both have zero
eigenvalues, the problem has the characteristic boundary,
where the characteristic variables have a vanishing speed
at the boundary \cite{Hil2013}. We first make a simultaneous
transformation of the coefficient matrices to reduce the 
original problem. Then following the characteristic analysis 
of ODEs, we explicitly write the general solution to the 
half-space problem. By analysing the stability of the 
solution, we finally achieve well-posed criteria for the 
linear non-homogeneous moment system in half-space. 
The method is similar as Kreiss' procedure \cite{Kreiss1970}
to study linear hyperbolic systems.
We find that the Grad boundary condition
does not obey the proposed well-posed rules
because of the instability of the non-homogeneous term. 
After minor modification, we can obtain well-posed boundary
conditions for the moment system \eqref{eq:und} in half-space.

A discrete system similar to \eqref{eq:und} comes from the
discrete velocity method (DVM) of the Boltzmann equation, 
where the matrix $\+A$ is diagonal. The solvability of 
these discrete equations has been exhaustively studied by
Bernhoff \cite{Bern2008,Bern2010,Bern2010b}. These results
can apply to the system \eqref{eq:und} to obtain
the existence and uniqueness of the solution. In comparison,
our method utilizes the specific structure of the Grad
moment equations, which gives a more subtle description of the 
solution. We also find the additional stability condition
besides the solvability condition. The general abstract 
theory about the linear boundary value problem can be found in 
\cite{1960Local,Kreiss1970,Osher1975,Rauch1985}.

The paper is arranged as follows: in Section 2, we briefly 
introduce the linearized Boltzmann equation and the Grad
moment equations. Then we state the well-posed conditions
without detailed proof. In Section 3, we illustrate the 
instability of the Grad boundary condition by simple examples
and discuss its well-posed modification. In Section 4, we 
complete the proof of well-posedness. The paper ends with 
a conclusion.

%% file: back.tex
\section{Basic Equations and Main Results}
\subsection{Basic Equations}
Around the equilibrium states, the rarefied gas can be described by
the linearized Boltzmann equation (LBE). For single-species monatomic 
molecules, we consider the LBE with the Maxwell boundary condition
\begin{subequations}
\begin{gather}
	\pd{f}{t} + \xi_d\pd{f}{x_d} = \mathcal{L}[f],\quad
	f=f(t,\+x,\+\xi),\ \+x\in\Omega\subset\bbR^3,\ \+\xi\in\bbR^3,\\
	\label{eq:Mbc}
	  f(t,\+x,\+\xi) = \chi f^w(t,\+x,\+\xi)+(1-\chi)f(t,\+x,\+\xi^*),
   \quad \+x\in\partial\Omega,\ (\+\xi-\+u^w) \cdot \+n < 0,
\end{gather}
\end{subequations}
where $f$ is the velocity distribution function, $t$ denoting the
time, $\+x=(x_1,x_2,x_3)$ the spatial coordinates and $\+\xi=(\xi_1,
\xi_2,\xi_3)$ the microscopic velocity. Here $\mathcal{L}$ is a linear
operator depicting the collision between gas molecules. 

The Maxwell model assumes that the boundary is an impermeable wall 
with a unit normal vector $\+n$ exiting the region, a given velocity
$\+u^w$, and a given temperature $\theta^w$. To avoid cumbersome details 
about the rotation invariance, we assume 
\[\Omega=\{\+x\in\bbR^3:\ x_2\geq 0\},\]
with a fixed $\+n=(0,-1,0)$. The reflection
at the wall is divided into a sum of $\chi$ portion of the specular 
reflection and $1-\chi$ portion of the diffuse reflection, 
where $\chi\in[0,1]$ is the tangential
momentum accommodation coefficient. To consider the steady layer
equations, we further assume
\[\+u^w\cdot\+n=0,\]
then the velocity from specular reflection is
\[\+\xi^* = \+\xi - 2[(\+\xi-\+u^w)\cdot \+n] \+n
=(\xi_1,-\xi_2,\xi_3),\]
and 
\begin{equation}
f^w(t,\+x,\+\xi) = \mathcal{M}(\+\xi)\left(
			\rho^w + \+u^w\cdot\+\xi + 
			\theta^w\frac{|\+\xi|^2-3}{2}\right).
\end{equation}
The reference distribution $\mathcal{M}$ is
a Maxwellian at rest
\[
	\mathcal{M}(\+\xi) = \frac{1}{(2\pi)^{3/2}}
	\exp\left(-\frac{|\+\xi|^2}{2} \right),
\]
where $\rho^w$ is determined by the no mass flow condition at the wall
\[
	\int_{\bbR^3}\!\!(\+\xi-\+u^w)\cdot\+n f(t,\+x,\+\xi)\,\mathrm{d}
	\+\xi = 0,\quad \+x\in\partial\Omega.
\]

The Grad moment equations with linearized ansatz assume that the 
distribution function has a Hermite expansion, for a chosen
integer $M\geq 2$, as
\begin{equation}\label{eq:lant}
	f(t,\+x,\+\xi) = \mathcal{M}(\+\xi)\sum_{|\+\alpha|\leq M}
	w_{\+\alpha}(t,\+x)\phi_{\+\alpha}(\+\xi),
\end{equation}
where $\+\alpha=(\alpha_1,\alpha_2,\alpha_3)\in\bbN^3$ is the
multi-index, $|\+\alpha|=\alpha_1+\alpha_2+\alpha_3$. 
The orthonormal Hermite polynomial 
$\phi_{\+\alpha}=\phi_{\+\alpha}(\+\xi)$ is defined
\cite{Grad1949N} by ensuring
\[\ang{\mathcal{M}\phi_{\+\alpha}\phi_{\+\beta}}=\delta_{\+\alpha,\+\beta},
	\quad \phi_{\+0}=1,\ \phi_{\+e_i}=\xi_i,\quad
	\ang{\cdot}:=\int_{\bbR^3}\!\!\cdot\,\mathrm{d}\+\xi,\]
where $\+e_i\in\bbN^3$ only has the $i$-th component being one. 
So the ansatz \eqref{eq:lant} gives 
\[w_{\+\alpha}=w_{\+\alpha}(t,\+x)
=\ang{f\phi_{\+\alpha}}.\] 
Substituting the ansatz into the 
linearized Boltzmann equation and matching the coefficients 
before basis functions, we have the $M$-th order moment equations
\begin{equation}\label{eq:G1}
	\pd{w_{\+\alpha}}{t} + \ang{\mathcal{M}\xi_d\phi_{\+\alpha}
	\phi_{\+\beta}}\pd{w_{\+\beta}}{x_d} =
	\ang{\mathcal{L}[\mathcal{M}\phi_{\+\beta}]\phi_{\+\alpha}}
	w_{\+\beta},\quad |\+\alpha|\leq M,
\end{equation}
where Einstein's convention is used to omit the summation 
notation about $|\+\beta|\leq M.$

Intuitively, the moment variables $w_{\+\alpha}$ can be related
to the macroscopic variables such as the density, the macroscopic 
velocity, and the temperature. So \eqref{eq:G1} would contain the
linearized Euler equations (around the equilibrium state given by
$\mathcal{M}$). We can regard \eqref{eq:G1} as an extension of 
hydrodynamic equations \cite{Grad1949}.

It's not difficult to check \cite{Zheng2021} that \eqref{eq:G1} is
symmetric hyperbolic. More precisely, the coefficient matrices are all
symmetric. To ensure the correct number of boundary
conditions for the hyperbolic system \cite{Hil2013}, 
Grad \cite{Grad1949} suggested
testing the Maxwell boundary condition \eqref{eq:Mbc} with odd
polynomials (about the argument $(\+\xi-\+u^w)\cdot\+n=-\xi_2$)
no higher than the $M$-th degree.

The Grad boundary condition has some equivalent representations
\cite{Grad1949,2008Linear,Cai2011,Lijun2017}. To exhibit the 
continuity of fluxes at the boundary, we extract the factor
$\xi_2$ of these odd polynomials and write the Grad boundary 
condition in an equivalent form:
\begin{eqnarray*}
	\int_{\bbR^2}\!\!\int_0^{+\infty}\!\!\xi_2
	\phi_{\+\alpha}f(t,\+x,\+\xi)\,\mathrm{d}\+\xi &=& 
	\chi\int_{\bbR^2}\!\!\int_0^{+\infty}\!\!\xi_2
	\phi_{\+\alpha}f^w(t,\+x,\+\xi)\,\mathrm{d}\+\xi +\\&&
	(1-\chi)\int_{\bbR^2}\!\!\int_0^{+\infty}\!\!\xi_2
	\phi_{\+\alpha}f(t,\+x,\+\xi^*)\,\mathrm{d}\+\xi,
\end{eqnarray*}
where $\alpha_2$ is even and $|\+\alpha|\leq M-1.$
Plugging the ansatz \eqref{eq:lant} into the above formula,
we can utilize the even-odd parity of Hermite polynomials
\cite{Yang2022a} to get
\begin{equation}\label{eq:rbc1}
	\left(1-\frac{\chi}{2}\right)\sum_{\beta_2 \text{\ odd}}\ang
	{\xi_2\mathcal{M}\phi_{\+\alpha}\phi_{\+\beta}} w_{\+\beta}
	= -\frac{\chi}{2}\sum_{\beta_2\text{\ even}}
	\ang{|\xi_2|\mathcal{M}
	\phi_{\+\alpha}\phi_{\+\beta}}(w_{\+\beta}-b_{\+\beta}),
\end{equation}
where the entries
$b_{\+0}=\rho^w,\ b_{\+e_i}=u_i^w,\ b_{2\+e_i}=\theta^w/\sqrt{2}$
and otherwise $b_{\+\alpha}=0.$ 

In conclusion, the $M$-th order linear Grad moment equations 
(in the plane geometry) are \eqref{eq:G1} with the Grad boundary
condition \eqref{eq:rbc1}, where $\rho^w$ is determined by the
no mass flow condition 
\[w_{\+e_2}-u_2^w=0,\quad \text{at}\ x_2=0.\]

Assume there is a boundary layer near $\partial\Omega=\{x_2=0\}.$
Then we may introduce the fast variables, which vary dramatically 
in the normal direction of the wall and vanish outside the boundary
layer. According to the asymptotic analysis \cite{Sone2007} of
\eqref{eq:G1}, these fast variables may satisfy the linear steady 
moment equations in half-space:
\[
	\ang{\mathcal{M}\xi_2\phi_{\+\alpha}
	\phi_{\+\beta}}\od{w_{\+\beta}}{y} =
	\ang{\mathcal{L}[\mathcal{M}\phi_{\+\beta}]\phi_{\+\alpha}}
	w_{\+\beta} + h_{\+\alpha},\quad |\+\alpha|\leq M,
\]
where $w_{\+\beta}=w_{\+\beta}(y),\ y\in[0,+\infty),$ represents
the fast variables, with $w_{\+\beta}(+\infty)=0$. The non-homogeneous
term $h_{\+\alpha}=h_{\+\alpha}(y)$ is given, arising from the other
contributions in the boundary layer.

The Grad boundary condition becomes
\begin{equation*}
	\left(1-\frac{\chi}{2}\right)\sum_{\beta_2 \text{\ odd}}\ang
	{\xi_2\mathcal{M}\phi_{\+\alpha}\phi_{\+\beta}} (w_{\+\beta}
	+ \bar{w}_{\+\beta})
	= -\frac{\chi}{2}\sum_{\beta_2\text{\ even}}
	\ang{|\xi_2|\mathcal{M}
	\phi_{\+\alpha}\phi_{\+\beta}}(w_{\+\beta}+\bar{w}_{\+\beta}
	-b_{\+\beta}),
\end{equation*}
where $\alpha_2$ is even with $|\+\alpha|\leq M-1,$ and
$\bar{w}_{\+\beta}$ is given by the bulk flow outside the boundary
layer.

\subsection{Main Results}
For the classical linearized Boltzmann operator \cite{CC1989},
we shall write the Grad moment equations in half-space as an 
abstract form
\begin{gather}\label{eq:NH1}
	\+A\od{W}{y} = -\+QW+h,\quad W=W(y),\ y\in[0,+\infty),\\
	W(+\infty)= 0,\notag
\end{gather}
where $\+A=\+A^T\in\bbR^{N\times N},\ \+Q\geq 0$ and $W\in\bbR^N.$
Here $N=\#\{\+\alpha\in\bbN^3:\ |\+\alpha|\leq M\}$ and we assume
\begin{equation}
	\mathrm{Null}(\+A)\cap\mathrm{Null}(\+Q)=\{0\}.
\end{equation}

In this paper, we will discuss the well-posed conditions
of \eqref{eq:NH1} with $h\neq 0,$ i.e., which boundary condition
should be prescribed at $y=0$ to ensure the well-posedness of
\eqref{eq:NH1}. Roughly, we only need to write general
solutions for the ODEs and analyze their stability.

If we can solve the generalized eigenvalue problem
of $(\+A,\+Q)$, we may deal with the characteristic equations
\[ \lambda_i\od{v_i}{y} = -v_i + h_i,\quad v_i(+\infty)=0,\]
where $\lambda_i\in\bbR\cup\{\infty\}$ and $h_i=h_i(y)$ is a 
given 1D function.

If $\lambda_i=0,$ we have $v_i=h_i$ and $v_i(0)=h_i(0)$, which
means that there is no need to prescribe extra boundary conditions
at $y=0.$ If $\lambda<0$, we have
\[ v_i(y) = e^{-\lambda_i^{-1}y}v_i(0) + \lambda_i^{-1}\int_0^y\!\!
e^{-\lambda_i^{-1}(y-s)}h_i(s)\,\mathrm{d}s.\]
So to make sure $v_i(+\infty)=0,$ there must be 
\begin{equation*}
v_i(0) = -\lambda_i^{-1}\int_0^{+\infty}\!\!e^{\lambda_i^{-1}s}
h_i(s)\,\mathrm{d}s,
\end{equation*}
if the integral exists. If $\lambda_i>0,$ from the
above formula we can see that $v_i(0)$ can be arbitrarily given
when
\[
\lim_{y\rightarrow +\infty}	\int_0^y\!\!
e^{-\lambda_i^{-1}(y-s)}h_i(s)\,\mathrm{d}s =0.
\]
Finally, if $\lambda_i=\infty$, formally $v_i$ should be a constant
and must be zero.

Our results are based on these simple observations. Due to the
structure of coefficient matrices, the role of the generalized
eigenvalue decomposition can be replaced by an 
interpretable simultaneous transformation of the matrices
$\+A$ and $\+Q$. We note that a similar transformation appears
in \cite{Bern2008,Bern2010b} in the language of projection operators.

For this purpose, we assume that $\+G\in\bbR^{N\times p}$ is 
the orthonormal basis matrix of $\mathrm{Null}(\+Q),$ and
$\+X\in\bbR^{p\times r}$ is the orthonormal basis matrix 
of $\mathrm{Null}(\+G^T\+A\+G),$ where
\[p=\mathrm{dim}\ \mathrm{Null}(\+Q),\quad 
r=\mathrm{dim}\ \mathrm{Null}(\+G^T\+A\+G).\] 
Since $\mathrm{Null}(\+A)\cap\mathrm{Null}(\+Q)=\{0\}$,
we have
	\[\mathrm{rank}(\+A\+G)=\mathrm{rank}(\+G)=p.\]
So we can let $\+V_1=\+G\+X\in\bbR^{N\times r}$ and construct  
$\+V_2\in\bbR^{N\times p}$ by a Gram-Schmidt orthogonalization
of $\+A\+G$. Then 
\[\mathrm{span}\{\+V_2\}=\mathrm{span}\{\+A\+G\}\]
and $\+V_2^T\+V_1=\+0.$ Let $\+V_3$ be the 
orthogonal complement of $[\+V_1,\+V_2]$.
So $\+V=[\+V_1,\+V_2,\+V_3]$ is orthogonal.
	
For technical reasons, we further define $\+U=[\+U_1,\+U_2,\+U_3]$,
where 
\[\+U_1=\+G\in\bbR^{N\times p},\
	\+U_2=\+A\+G\+X\in\bbR^{N\times r},\ \+U_3=\+V_3.\]
We claim that $\+U$ is invertible. Since $\+V_3^T\+V_2=\+0,$ 
we have $\+V_3^T\+U_2=\+0$. To make $\+U$ invertible, it's enough
to show that 
\begin{equation}\label{eq:v3g}
	\mathrm{span}\{\+V_3\}\cap\mathrm{span}\{\+G\}=\{0\}.
\end{equation}
In fact, if $\+G c_1=\+V_3 c_2$ for some $c_1\in\bbR^p$ and $
c_2\in\bbR^{(N-p-r)}$, then 
	$(\+G c_1)^T\+A\+G= 0$ since $\+V_3^T\+V_2=\+0$.
So $c_1\in\mathrm{span}\{\+X\}$ and  
$\+V_3 c_2\in\mathrm{span}\{\+G\+X\}=\mathrm{span}\{\+V_1\}$. 
Thus, from $\+V_3^T\+V_1=\+0$ we have $c_2=0$, which 
implies (\ref{eq:v3g}).

\begin{lemma}\label{lem:01}
	The transformed matrices $\+A_{ij}=\+U_i^T\+A\+V_j$ and 
	$\+Q_{ij}=\+U_i^T\+Q\+V_j$ satisfy
	\begin{itemize}
		\item $\+Q_{1j}=\+0,\ \+Q_{i1}=\+0,\ \+Q_{33}>0,$ for $i,j=1,2,3.$
		\item $\+A_{31}=\+0,\ \+A_{33}^T=\+A_{33}$ and
			$\mathrm{rank}(\+A_{21})=\mathrm{rank}(\+U_2).$
	\end{itemize}
\end{lemma}
\begin{proof}
	By definition, $\+A_{33}$ is symmetric. 
	Since $\+Q\+G=\+0,$ we have $\+Q_{1j}=\+0$
	and $\+Q_{i1}=\+0.$ From \eqref{eq:v3g}, we have $\+Q_{33}>0.$
	We have $\mathrm{rank}(\+A_{21}) =
	\mathrm{rank}(\+U_2^T\+U_2) = \mathrm{rank}(\+U_2)$. 
	Meanwhile, $\+A_{31}=\+V_3^T\+U_2=\+0$. 
\end{proof}

Thanks to the above transformation, we may consider the 
generalized eigenvalue problem of
$(\+A_{33},\+Q_{33})$ rather than $(\+A,\+Q),$ where $\+Q_{33}$
is symmetric positive definite while $\+Q$ is symmetric positive
semi-definite. Here $\+A_{33}$ and $\+A$ are both symmetric 
matrices which may have zero eigenvalues. 
The choices of $\+V_2$ and $\+V_3$ are not unique, but they
do not affect the results of Lemma \ref{lem:01}.

The general solutions of ODEs should be written with the aid
of the eigenvalue decomposition.
According to Sylvester's law of inertia, $\+Q_{33}^{-1}\+A_{33}$ 
should have the same number of positive, negative and zero 
eigenvalues as $\+A_{33}$. Assume the Cholesky decomposition
\[\+Q_{33}=\+L\+L^T,\]  
then we must have an orthogonal eigenvalue decomposition of the
real symmetric matrix:
	\begin{equation}\label{eq:LAL}
		\+L^{-1}\+A_{33}\+L^{-T}\+R=\+R\+\Lambda.
	\end{equation}
We assume the diagonal matrix
\[
\+\Lambda = \begin{bmatrix}\+\Lambda_+ &
	& \\ & \+0 & \\ & & \+\Lambda_{-}\end{bmatrix},
\]
where $\+\Lambda_+\in\bbR^{n_+\times n_+}$ has postive entries
and $\+\Lambda_-\in\bbR^{n_-\times n_-}$ has negative entries. 
The matrix $\+R=[\+R_+,\+R_0,\+R_-]$ is assumed orthogonal, where the
columns of $\+R_+,\+R_0,\+R_-$ are individually the number of
postive, zero and negative eigenvalues of $\+A_{33}$, i.e.,
$n_+,n_0,n_-$. Suppose
\[\+T=\+L^{-T}\+R=[\+T_+,\+T_0,\+T_-],\]
where the columns of $\+T_+,\+T_0,\+T_-$ are individually 
$n_+,n_0,n_-$. Then $\+T$ is invertible and we have 
\begin{equation}\label{eq:T0}
	\+Q_{33}^{-1}\+A_{33}[\+T_+,\+T_{0},\+T_-] = 
	[\+T_+,\+T_{0},\+T_-]\begin{bmatrix}\+\Lambda_+ &
	& \\ & \+0 & \\ & & \+\Lambda_{-}\end{bmatrix}.
\end{equation}

For a positive number $a$, we define the norm in $L^2(\bbR_+,
L^2(\bbR^N))$ as 
\[ \|h\|_a = \left(\int_0^{+\infty}\!\!e^{2ay}h^Th(y)
\,\mathrm{d}y\right)^{1/2},\]
and the vector norm as 
\[\|g\| = \sqrt{g^Tg}.\]
Then we have the following well-posed theorem, whose complete
proof is put in Section \ref{sec:4}.

\begin{theorem}
	\label{thm:01}
	Assume $a>0$ is a given constant such that $a<1/\lambda_{max}$,
	where $\lambda_{max}$ is the maximal eigenvalue of 
	$\+Q_{33}^{-1}\+A_{33}$. We propose the boundary condition
	\begin{equation}
		\label{eq:B3bc}
		\+B\+V_3^TW(0) = g
	\end{equation}
	for the system \eqref{eq:NH1}, where $\+B\in\bbR^{n_+\times 
	(n_++n_0+n_-)}$ and $g\in\bbR^{n_+}$ are given,
	with constant coefficients.

	For any $h$ with $\|h\|_a<+\infty$ and $g$ with $\|g\|<+\infty$,
	the system \eqref{eq:NH1} has a unique solution and there exists
	a positive constant $C_{a,M}$ independent of the non-homogeneous
	term $h$ such that the following estimation holds:
	\begin{equation}\label{eq:est} 
		\|W\|_a \leq C_{a,M}\left(\|h\|_a + \|g\|\right),
	\end{equation}
	if and only if the boundary condition satisfies the following
	conditions
	\begin{enumerate}
		\item $\mathrm{rank}(\+B\+T_+)=n_+.$
		\item $\+B\+T_0=\+0.$
	\end{enumerate}
	
	What's more, the analytical
	expressions of the solution are given by 
	\eqref{eq:V1T_n}, \eqref{eq:V2T_n} and \eqref{eq:V3T_n}.
\end{theorem}

\begin{remark}
	According to the proof in Section \ref{sec:4}, we can write
	$\+V_3^T W(0)$ as
	\[\+V_3^T W(0) = \+T_+ z_+(0) + \+T_{0} z_0(0) + \+T_- z_-(0),\]
	where $z_0$ and $z_-$ are determined by $h$. So the first
	condition in Theorem \ref{thm:01} makes sure the unique
	solvability of $z_+(0)\in\bbR^{n_+}$. The second condition
	shows that $z_+(0)$ will not be affected by $z_0(0)$, i.e.,
	\[z_+(0)=(\+B\+T_+)^{-1}(g-\+B\+T_-z_-(0)).\]
	Because $\|z_0(0)\|$ can not be controlled by $\|g\|$ and
	$\|h\|_a$, removing it from the solution ensures the estimation
	\eqref{eq:est}. In the homogeneous case, we have $h=0$ and
	$z_0=0,\ z_-=0$, which means that there is no need to consider 
	the second condition.
\end{remark}

\begin{remark}
	In Theorem \ref{thm:01}, we restrict the shape of the coefficient
	matrix $\+B$ to be $n_+\times (n_++n_0+n_-)$. Then since 
	$\+T_+\in\bbR^{(n_++n_0+n_-)\times n_+}$, the first condition
	in Theorem \ref{thm:01} shows that $\+B\+T_+$ is invertible.
	The assumption is
	mainly for ease of exposition. If we give the boundary condition
	\[\+B_3\+V_3^TW(0) = \tilde{g},\]
	where $\+B_3\in\bbR^{\tilde{n}\times (n_++n_0+n_-)}$, then the
	solvability requires $\mathrm{rank}(\+B_3\+T_+)=n_+$ and 
	\begin{equation}\label{eq:2e}
		\tilde{g}-\+B_3\+T_0z_0(0)-\+B_3\+T_-z_-(0)\in\mathrm{span}\{
		\+B_3\+T_+\}.
	\end{equation}
	In that case, $\tilde{g}$ may depend on $h$ due to the condition
	\eqref{eq:2e} and $\|\tilde{g}\|$ may not have a upper bound. 
	So it's a little wordy to obtain an analogous 
	estimation as \eqref{eq:est}. The key point is that if we can 
	find a matrix
	$\+C\in\bbR^{\tilde{n}\times n_+}$ such that 
	$\+C^T\+B_3\+T_+$ is invertible and $\+C^T\+B_3\+T_0=\+0$,
	then Theorem \ref{thm:01} helps to give the estimation.
\end{remark}

%% file: draw.tex
\section{Modified Boundary Conditions}
\subsection{Instability of the Grad Boundary Condition}
For initial-boundary value problems of the linear hyperbolic 
system, Majda and Osher \cite{Osher1975} emphasize that the 
linear space determined by the boundary condition
should contain the null space of the boundary matrix.
Otherwise, the boundary condition may be unstable.
Based on the similar observation, \cite{Sarna2018} points out
that the linear Grad boundary condition is unstable for the 
initial-boundary value problem. 

For the half-space problem, we will use a simple example to 
illustrate the instability of the Grad boundary condition. 
From the example, we can see that the instability not only
comes from the half-space problem.

If we consider Kramers' problem with the BGK
collision term \cite{Lijun2017}, the simplest moment system when
$M=3$ can read as
\begin{gather}\notag
\od{\sigma_{12}}{y} = 0,\\ \label{eq:K3n}
\od{u_1}{y} + \sqrt{2}\od{f_3}{y}
= -\nu\sigma_{12},\\ \notag
\sqrt{2}\od{\sigma_{12}}{y} = 
-\nu f_3 + h_3,
\end{gather}
where $\nu>0$ is a constant and we write the moment variables
as $\sigma_{12},u_1,f_3$. The term $h_3=h_3(y)$ is the given
non-homogeneous term. The Grad boundary condition reads as
\begin{equation}\label{eq:b01}
\bar{\sigma}+\sigma_{12} + \hat{\chi}\left(u_1+\bar{u}
+\frac{\sqrt{2}}{2}f_3\right) = 0,
\end{equation}
where $\hat{\chi}=\displaystyle
\frac{2\chi}{2-\chi}\frac{1}{\sqrt{2\pi}}$ and
$\bar{\sigma},\ \bar{u}$ are given by the far field.
Since the moment variables vanish when $y=+\infty$, we can 
solve from \eqref{eq:K3n} that 
\begin{equation}\label{eq:sol3}
	\sigma_{12}=0,\ u_1=-\sqrt{2}f_3,\ f_3=h_3/\nu.
\end{equation}
So the Grad boundary condition gives
\[
\bar{u} = -\frac{1}{\hat{\chi}}\bar{\sigma}+\frac{\sqrt{2}}{2}
\frac{h_3(0)}{\nu}.
\]
Only when $\bar{u}$ and $\bar{\sigma}$ satisfy the above relation
can the half-space problem has a unique solution. 
However, the term $h_3(0)$ can not be controlled by the weighted
$L^2$ norm $\|h_3\|_a$. In this sense, we claim that $\bar{u}$ is
unstable because $\|\bar{u}\|$ can not be controlled by $\|h_3\|_a$. 

Using the notations in Theorem \ref{thm:01}, in this example,
we have 
\[ \+A = \begin{bmatrix} 0 & 0 & 1 \\
		0 & 0 & \sqrt{2} \\
		1 & \sqrt{2} & 0 \end{bmatrix},\
	\+Q = \begin{bmatrix} 0 & & \\ & \nu & \\ & & \nu \end{bmatrix},\
	h = \begin{bmatrix} 0 \\ h_3 \\ 0 \end{bmatrix},\
	W = \begin{bmatrix}u_1\\f_3 \\ \sigma_{12}\end{bmatrix}.\]
The previous procedure will give
\[
\+V_1=\begin{bmatrix}1 \\ 0\\ 0 \end{bmatrix},\
   \+V_3=\begin{bmatrix}0 \\ 1\\ 0 \end{bmatrix},\
	   n_+=0,\ n_0=1.
\]
Since $n_+=0,$ Theorem \ref{thm:01} tells that the half-space
problem does not need boundary conditions at $y=0$, where the moment
system directly gives the solution \eqref{eq:sol3}. But the Grad
boundary condition gives one more condition. So this condition
asks for relations of the given values, e.g., $\bar{u}$ and
$\bar{\sigma}$. We should also consider the stability of these
additional solutions.

When $M=5$, we can check that $n_+>0$ and the Grad boundary 
condition does not satisfy 
\[\+B\+T_0=\+0,\]
which shows the instability of the Grad boundary condition
from Theorem \ref{thm:01}. An analogous example is given
in \cite{Sarna2018}.

To overcome this drawback, one may change \eqref{eq:b01} as
\[
c(\bar{\sigma}+\sigma_{12}) + \hat{\chi}\left(u_1+\bar{u}
+\sqrt{2}f_3\right) = 0,
\]
where $c>0$ is an arbitrary positive constant. Then the modified
boundary condition will give
\[
\bar{u} = -\frac{1}{c\hat{\chi}}\bar{\sigma},
\]
which is stable about $h_3$. The boundary conditions 
involving higher-order moment variables can be modified
in the same way.
Authors of \cite{Sarna2018} suggest choosing $c=1$, such that
the modified boundary condition differs from the Grad boundary
condition only by the coefficients before the highest order
moment variables, i.e., $f_3$ in this example. However, it's
not clear whether the modification is optimal in the sense
of providing the most accurate solutions.

\subsection{Modified Boundary Conditions}
We systematically introduce a class of modified boundary
conditions and prove their well-posedness. Assume
\[\mathbb{I}_e=\{\+\alpha\in\bbN^3:\ \alpha_2\text{\ even},\ 
|\+\alpha|\leq M\},\quad
\mathbb{I}_o=\{\+\alpha\in\bbN^3:\ \alpha_2\text{\ odd},\ 
|\+\alpha|\leq M\},\]
and $m=\#\mathbb{I}_e,\ n=\#\mathbb{I}_o.$ So we have $m+n=N$
and $m\geq n.$ Suppose the multi-indices with the even second
component are always ordered before the ones with the odd
second component, e.g., $(a_1,0,a_3)$ is ordered before
$(b_1,1,b_3)$ for any $a_1,a_3,b_1,b_3$. Then we can write the
Grad boundary condition \eqref{eq:rbc1} as 
\begin{equation}\label{eq:mbc1}
\+E\+M(W_o-b_o)+\hat{\chi}\+E\+S(W_e-b_e)=0,
\end{equation}
where $\hat{\chi}=\displaystyle
\frac{2\chi}{2-\chi}\frac{1}{\sqrt{2\pi}}$. The matrix
$\+M\in\bbR^{m\times n}$ has entries 
$\ang{\mathcal{M}\xi_2\phi_{\+\alpha}\phi_{\+\beta}}$
for $\+\alpha\in\mathbb{I}_e$ and $\+\beta\in\mathbb{I}_o.$
And the matrix $\+S\in\bbR^{m\times m}$ has entries
$\ang{\mathcal{M}|\xi_2|\phi_{\+\alpha}\phi_{\+\beta}}$
for $\+\alpha,\ \+\beta\in\mathbb{I}_e$.
The variables are divided as
\[
	W=\begin{bmatrix} W_e \\ W_o \end{bmatrix},\ 
	b=\begin{bmatrix} b_e \\ b_o \end{bmatrix},
\]
where the elements of $W_e\in\bbR^m$ are $w_{\+\alpha},\
\+\alpha\in\mathbb{I}_e$ and the elements of $W_o\in\bbR^n$
are $w_{\+\alpha},\ \+\alpha\in\mathbb{I}_o.$
Every row of $\+E\in\bbR^{n\times m}$ is a unit vector with 
only one component being one, such that the entries of $\+E\+M$ 
are 
$\ang{\mathcal{M}\xi_2\phi_{\+\alpha}\phi_{\+\beta}}$
$,\+\alpha\in\mathbb{I}_e,\ |\+\alpha|\leq M-1,$ 
and $\+\beta\in\mathbb{I}_o.$

Due to the recursion relation and orthogonality of Hermite 
polynomials \cite{Fan_new}, we can write 
\[\+A=\begin{bmatrix} \+0 & \+M \\ \+M^T & \+0\end{bmatrix},
	\quad \+Q = \begin{bmatrix} \+Q_e & \+0\\ \+0 & \+Q_o
	\end{bmatrix},\]
where $\+Q_e\in\bbR^{m\times m}$ and $\+Q_o\in\bbR^{n\times n}.$
By definition, we also know that $\+M$ is of full column rank 
and $\+S$ is symmetric positive definite \cite{Yang2022a}.

The modified boundary condition for the initial-boundary
value problem is
\begin{equation}\label{eq:mbc2}
\+H(W_o-b_o)+\hat{\chi}\+M^T(W_e-b_e)=0,
\end{equation}
where $\+H\in\bbR^{n\times n}$ is an arbitrary symmetric positive
definite matrix. For the half-space problem, the boundary condition
should write as 
\begin{equation}\label{eq:mbc3}
\+H(W_o+\bar{W}_o-b_o)+\hat{\chi}\+M^T(W_e+\bar{W}_e-b_e)=0,
\end{equation}
where $\bar{W}_e$ and $\bar{W}_o$ are given by the flow outside
the boundary layer.

Inspired by the illustrative examples, we can first find a 
matrix $\+C\in\bbR^{n\times n_+}$ and multiply \eqref{eq:mbc3}
left by $\+C^T$. Then, we can use Theorem \ref{thm:01} to check
the well-posedness of the half-space problem. Finally, we may
solve the remaining part of \eqref{eq:mbc3} to obtain
relations between the given vectors, e.g., $\bar{W}_o$ and
$\bar{W}_e$.

Following this way, we first calculate the value of $n_+$
by the special block structure of $\+A$ and $\+Q$.
We may as well write
	\[
		\+G = \begin{bmatrix} \+G_e & \\ & \+G_o \end{bmatrix},\
		\+X = \begin{bmatrix} \+X_e & \\ & \+X_o \end{bmatrix},\
	\]
	where $\+G_e\in\bbR^{m\times p_1}$, $\+G_o\in\bbR^{n\times p_2}$,
	$\+X_e\in\bbR^{m\times r_1}$, $\+X_o\in\bbR^{n\times r_2}$. And
	\[
		\+V_1=\begin{bmatrix} \+Y_1 & \\ & \+Z_1 \end{bmatrix},\	
		\+V_2=\begin{bmatrix} \+Y_2 & \\ & \+Z_2 \end{bmatrix},\
		\+V_3=\begin{bmatrix} \+Y_3 & \\ & \+Z_3 \end{bmatrix},	
	\]
	where $\+Y_1=\+G_e\+X_e,\ \+Z_1=\+G_o\+X_o$, 
	\[\mathrm{span}\{\+Y_2\}=\mathrm{span}\{\+M\+G_o\},\
	\mathrm{span}\{\+Z_2\}=\mathrm{span}\{\+M^T\+G_e\},\]
	$\+Y_3\in\bbR^{m\times (m-r_1-p_2)}$ 
	and $\+Z_3\in\bbR^{n\times (n-r_2-p_1)}$.
	Under these assumptions, we claim that 
	
	\begin{lemma}\label{lem:12}
		$n_+=n-r_2-p_1.$
	\end{lemma}

	\begin{proof}
	We may as well write the Cholesky 
	decomposition $\+Q_{33}=\+L\+L^T$ as  
		\[\+L=\begin{bmatrix}\+L_e & \\ & \+L_o\end{bmatrix},\
	\+L_e\in\bbR^{(m-r_1-p_2)\times (m-r_1-p_2)},\
		\+L_o\in\bbR^{(n-r_2-p_1)\times (n-r_2-p_1)}.\] 
		Then we have
	\[
		\+L^{-1}\+A_{33}\+L^{-T} = \begin{bmatrix}
		\+0 & \+L_e^{-1}\+Y_3^T\+M\+Z_3\+L_o^{-T} \\ 
		\+L_o^{-1}\+Z_3^T\+M^T\+Y_3\+L_e^{-T} & \+0 
		\end{bmatrix}.
	\]
	We will show that 
		$\+Y_3^T\+M\+Z_3$ is of full column rank,
	and the positive as well as negative 
	eigenvalues of $\+L^{-1}\+A_{33}\+L^{-T}$ should appear in 
	pair. These facts will lead to $n_+=n-r_2-p_1.$

	Suppose  $\+Y_3^T\+M\+Z_3 x = 0$ for some 
	$x\in\bbR^{n-r_2-p_1}$. Since $[\+Y_1,\+Y_2,\+Y_3]$
	is orthogonal, there exists $x_1\in\bbR^{r_1}$ and
	$x_2\in\bbR^{p_2}$ such that
\begin{equation}\label{eq:tem2}
\+M\+Z_3 x = \+Y_1 x_1+\+Y_2 x_2. 
\end{equation}
Since $\+Z_3^T\+M^T\+Y_1=\+0$ and $\+Y_2^T\+Y_1=\+0$, we 
have $x_1=0.$ The relation \eqref{eq:v3g} implies that 
\[\mathrm{span}\{\+Z_3\}\cap\mathrm{span}\{\+G_o\}=\{0\}.\]
So there must be $x=0$ 
since $\+M$ and $\+Z_3$ are of full column rank. This shows that
$\+Y_3^T\+M\+Z_3$ is of full column rank.

For eigenvalues and eigenvectors of the symmetric matrix,
we introduce a general conclusion.
Let $\+D \in \bbR^{\alpha \times \beta}$ and
$\mathrm{rank}(\+D)=\gamma$. Then 
		\[\tl{D}:=\begin{bmatrix} \+0 & \+D \\ \+D^T & \+0
		\end{bmatrix}\]
has $\alpha + \beta - 2\gamma$ zero eigenvalues,
$\gamma$ positive eigenvalues and $\gamma$ negative eigenvalues.
	
	In fact, the symmetric matrix $\tl{D}$ must have $\alpha+\beta$ real
	eigenvalues. Assume $\lambda\in\bbR$ is an eigenvalue, 
	then there exists $x\in\bbR^{\alpha}$ and 
	$y\in\bbR^{\beta}$ such that 
	\[\begin{bmatrix} \+0 & \+D \\ \+D^T & \+0 \end{bmatrix}
		\begin{bmatrix} x \\ y \end{bmatrix} = \lambda
	\begin{bmatrix} x \\ y \end{bmatrix} \quad\Rightarrow
		\quad \begin{bmatrix} \+0 & \+D \\ \+D^T & \+0
	   	\end{bmatrix} \begin{bmatrix} x \\ -y \end{bmatrix}
	   	= -\lambda \begin{bmatrix} x \\ -y \end{bmatrix}.\]
		So $-\lambda$ is also an eigenvalue, which implies that 
		$\tl{D}$ has the same number of positive and negative 
		eigenvalues. Since $\mathrm{rank}(\+D)
			=\mathrm{rank}(\+D^T)=\gamma$, there must be
			$\alpha+\beta-2\gamma$ zero eigenvalues.

		Applying the above result to $\+L^{-1}\+A_{33}\+L^{-T}$,
			we then have $n_+=n-r_2-p_1.$
\end{proof}

	Incidentally, applying the above result in Lemma \ref{lem:12} to
	$\+G^T\+A\+G$, we have $r_1+p_2=r_2+p_1.$
	Then the required matrix $\+C\in\bbR^{n\times n_+}$ can be
	found.

\begin{lemma}
	\label{lem:11}
	Let $\+B_3=[\hat{\chi}\+M^T,\+H]\+V_3$ and $\+C=\+Z_3$.
	Then when $\chi\in[0,1]$, we have 
	\begin{itemize}	
		\item[i.] $\+C^T\+B_3\+T_+$ is invertible.
		\item[ii.] $\+C^T\+B_3\+T_0=\+0.$
	\end{itemize}
\end{lemma}

\begin{proof}
	Thanks to the block structure of $\+L^{-1}\+A_{33}\+L^{-T}$,
	we can assume 
	\[\+R_+ = \begin{bmatrix}
	\+R_e\\ \+R_o \end{bmatrix},\quad \+R_e\in\bbR^{(m-r_1-p_2)\times
		n_+},\quad \+R_o\in\bbR^{n_+\times n_+},\]
	where $\displaystyle 
	\+R_e^T\+R_e=\+R_o^T\+R_o=\frac{1}{2}\+I_{n_+}$ 
	with the identity matrix $\+I_{n_+}$. We have 
\begin{eqnarray*}
	\+Z_3^T\+B_3\+T_+  
	&=& \hat{\chi}\+Z_3^T\+M^T\+Y_3\+L_e^{-T}\+R_e + 
	\+Z_3^T\+H\+Z_3\+L_o^{-T}\+R_o \\
	&=& \hat{\chi}\+L_o\+R_o\+\Lambda_+ +
	\+Z_3^T\+H\+Z_3\+L_o^{-T}\+R_o.
\end{eqnarray*}
Since $\hat{\chi}\geq 0$ when $\chi\in[0,1]$, for any
$x\in\bbR^{n_+}$, we have
\begin{eqnarray}\notag
	x^T\+R_o^T\+L_o^{-1}\+Z_3^T\+B_3\+T_+ x &=&
	\frac{1}{2}\hat{\chi}x^T\+\Lambda_+ x
	+ x^T\+R_o^T\+L_o^{-1}\+Z_3^T\+H\+Z_3\+L_o^{-T}\+R_o x \\
	&\geq& x^T\+R_o^T\+L_o^{-1}\+Z_3^T\+H\+Z_3\+L_o^{-T}\+R_o x
	\geq 0, 
	\label{eq:tm33}
\end{eqnarray}
where the equality holds if and only if $x=0.$ This shows that 
$\+R_o^T\+L_o^{-1}\+Z_3^T\+B_3\+T_+$ is symmetric positive
definite. So $\+Z_3^T\+B_3\+T_+$ is invertible. 

   	Then we show 
	$\+Z_3^T\+B_3\+T_0=\+0$.
By definition, we have $\+Q_{33}^{-1}\+A_{33}\+T_0=\+0,$
which shows that $\+A_{33}\+T_0=\+0.$ 
Due to the block structure of $\+A_{33}$, the matrix $\+T_0$
can write as 
\[\+T_0=\begin{bmatrix} \+T^* \\ \+0 \end{bmatrix},\ 
	\+T^*\in\bbR^{(m-r_1-p_2)\times n_0},\]
since $\+Y_3^T\+M\+Z_3$ is of full column rank.
This gives $\+Z_3^T\+M^T\+Y_3\+T^*=\+0$ and we have
\[ 
	\+Z_3^T\+B_3\+T_0 = \hat{\chi}\+Z_3^T\+M^T\+Y_3\+T^*=\+0,
\]
which completes the proof.
\end{proof}

With the aid of the above lemmas, we state the well-posedness
theorem as follows:

\begin{theorem}
	Suppose $\hat{\chi}>0$ and $r_2=0$. For any $h$ with 
	$\|h\|_a<+\infty$ and $g_1=\bar{W}_o-b_o$,\
	$g_2=(\+I_m-\+G_e\+G_e^T)(\bar{W}_e-b_e)$ with $\|g_1\|<+\infty,
	\ \|g_2\|<+\infty$, the moment system \eqref{eq:NH1} with the
	boundary condition \eqref{eq:mbc3} has
	a unique solution of $W$ and $\+G_e^T(\bar{W}_e-b_e)$, and the 
	solution satisfies the estimation
	\begin{gather*}
		\|W\|_a \lesssim \|h\|_a+\|g_1\|+\|g_2\|,\\
		\|\+G_e^T(\bar{W}_e-b_e)\| \lesssim \|h\|_a+\|g_1\|+\|g_2\|,
	\end{gather*}
where $a\lesssim b$ represents that
there exists a constant $c>0$ such that $a\leq cb.$
\end{theorem}

\begin{proof}
	The assumption $r_2=0$ makes the matrices $\+X_o$ and 
	$\+Z_1$ vanish, which simplifies the proof. Note that
	$\+V$ is orthogonal. Plugging $\+V\+V^T$ before the variables
	$W_e$ and $W_o$ in \eqref{eq:mbc3}, and decomposing $\bar{W}_e-b_e$
	as
	\[\bar{W}_e-b_e = \+G_e\+G_e^T(\bar{W}_e-b_e) + 
		(\+I_m-\+G_e\+G_e^T)(\bar{W}_e-b_e),
	\]
	the boundary condition can be equivalently written as 
	\begin{eqnarray}
		&&\hat{\chi}\+M^T\left(\+Y_1\+Y_1^TW_e(0)+\+G_e\+G_e^T
		(\bar{W}_e-b_e)\right) \notag 
		+ \+B_3\+V_3^TW(0) \\ &=&
		-\+H(\bar{W}_o-b_o)-\hat{\chi}\+M^T(\+I_m-\+G_e\+G_e^T)
				(\bar{W}_e-b_e) \label{eq:0923}
				-[\hat{\chi}\+M^T,\+H]\+V_2\+V_2^TW(0), 
	\end{eqnarray}
	where $\+B_3=[\hat{\chi}\+M^T,\+H]\+V_3$ as in Lemma \ref{lem:11}.

	We extract the lower block of $\+U$ as
	$\tl{U}=[\+G_o,\+M^T\+G_e\+X_e,\+Z_3]\in\bbR^{n\times n}.$
	Then $\tl{U}$ is invertible since $\+U$ is invertible.
	Multiplying \eqref{eq:0923} left by $\tl{U} ^T$, the boundary
	condition can be divided into three parts.

	The first part corresponds to multiply \eqref{eq:0923} left
	by $\+Z_3^T$. Since $\mathrm{span}\{\+M^T\+G_e\}
	=\mathrm{span}\{\+Z_2\}$ and $\+Y_1=\+G_e\+X_e$, we have
	$\+Z_3^T\+M^T\+Y_1=\+Z_3^T\+M^T\+G_e=\+0$ because
	$\+Z_3^T\+Z_2=\+0.$ Now the boundary condition gives
	\begin{eqnarray}\label{eq:sbc1}
		\+Z_3^T\+B_3\+V_3^TW(0)=
	   	-\+Z_3^T\+H g_1-\hat{\chi}\+Z_3^T\+M^T g_2 
				-\+Z_3^T[\hat{\chi}\+M^T,\+H]\+V_2\+V_2^TW(0),
	\end{eqnarray}
	which satisfies the well-posed conditions in Theorem \ref{thm:01}
	due to Lemma \ref{lem:11}.
	From the proof of Theorem \ref{thm:01}, we know
	\[\|\+V_2^TW(0)\|\lesssim \|h\|_a.\]
	So from Theorem \ref{thm:01}, the half-space system \eqref{eq:NH1}
	with the boundary condition \eqref{eq:sbc1} has a unique
	solution of $W$ and 
	\begin{gather*}
		\|W\|_a \lesssim \|h\|_a+\|g_1\|+\|g_2\|.
	\end{gather*}

	The remaining part of the boundary condition \eqref{eq:0923}
	should be compatible. Note that $r_1+p_2=r_2+p_1=p_1.$ 
	We write $\+U_0=[\+G_o,\+M^T\+G_e\+X_e]\in\bbR^{n\times p_1}$.
	Then the remaining part of the boundary condition corresponds
	to multiply \eqref{eq:0923} left by $\+U_0^T.$ We claim that
	\[\mathrm{rank}(\+U_0^T\+M^T\+G_e)=p_1.\]
	Otherwise, there exists $x\in\bbR^{p_1}$ such that 
	$x\in\mathrm{span}\{\+U_0\}$ and $x^T\+M^T\+G_e=0.$
	Since $\+Z_1$ vanishes and $[\+Z_2,\+Z_3]$ is orthogonal, we have
	$x\in\mathrm{span}\{\+Z_3\}$. But $[\+U_0,\+Z_3]$ is invertible.
	So $x=0$ and $\mathrm{rank}(\+U_0^T\+M^T\+G_e)=p_1.$

	Thus, when $W$ is known, we can solve a unique
	$\+G_e^T(\bar{W}_e-b_e)$ by multiplying \eqref{eq:0923} left
	with $\+U_0^T.$ From the proof of Theorem \ref{thm:01},
	only the trace $z_0(0)$ included in $\+V_3^TW(0)$ can not be
	controlled by $\|h\|_a$.
	Fortunately, from \eqref{eq:V1T_n}, the term
	$\+A_{21}\+V_1^TW(0)+\+A_{23}\+V_3^TW(0)$ 
	will not contain $z_0(0)$. So we have the estimation
	\[
		\|\+G_e^T(\bar{W}_e-b_e)\|\lesssim \|g_1\|+\|g_2\|
		+\|h\|_a.
	\]

	This completes the proof.
\end{proof}

In conclusion, the boundary condition for the moment equations
in half-space is different from the case of initial-boundary value 
problems. Only when the variables given by the outside flow, e.g.,
$\bar{W}_e$ and $\bar{W}_o,\ $satisfy
some relations, can the half-space problem has a unique solution.
What's more, in the initial-boundary value problem, the boundary
condition \eqref{eq:mbc2} should be equipped with the no mass
flow condition $w_{\+e_2}=0$ to determine the extra variable
$\rho^w$ in $b_e$. While in the half-space problem, $w_{\+e_2}=0$
is automatically ensured by the system \eqref{eq:NH1} and
we do not need extra conditions at $y=0$ other than \eqref{eq:mbc3}.

Another possible choice of $\+H$ is $\+H=\+M^T\+S^{-1}\+M$.
The physical meaning of this construction is imposing 
alternative continuity of fluxes at the boundary.
The condition can be put into the framework of Grad's, i.e.,
testing the Maxwell boundary condition by some odd polynomials.
But in this case the test functions are different from the ones
in the Grad boundary condition. To see this, we first test the
Maxwell boundary condition with $\xi_2\phi_{\+\alpha},
\+\alpha\in\mathbb{I}_e$ to get 
\[
\+M(W_o+\bar{W}_o-b_o)+\hat{\chi}\+S(W_e+\bar{W}_e-b_e)=0.
\]
The test functions in the Grad boundary condition are
$\xi_2\phi_{\+\alpha}, \+\alpha\in\mathbb{I}_e, |\+\alpha|\leq M-1.$ 
Alternatively, we can combine the polynomials $\xi_2\phi_{\+\alpha},
\+\alpha\in\mathbb{I}_e,$ linearly to get
\[
\+M^T\+S^{-1}\+M(W_o+\bar{W}_o-b_o)+\hat{\chi}\+M^T(W_e+\bar{W}_e-b_e)=0,
\]
which gives $\+H=\+M^T\+S^{-1}\+M.$ The construction
is convenient to use for arbitrary order moment equations.

%% file: proof.tex
\section{Details in the Proof}
\label{sec:4}
The section aims to fill in the missing details in the proof
of Theorem \ref{thm:01}. The proof mainly contains two steps,
where we first explicitly write the analytical solution
of \eqref{eq:NH1} and then estimate it.

\begin{lemma} 
	\label{lem:41}
	Under the assumption of Theorem \ref{thm:01}, the system
	has a unique solution given by \eqref{eq:V1T_n}, \eqref{eq:V2T_n}
	and \eqref{eq:V3T_n}, if and only if 
		\[ \mathrm{rank}(\+B\+T_+)=n_+.\]
\end{lemma}
\begin{proof}
	Using Lemma \ref{lem:01}, the system \eqref{eq:NH1} is 
	equivalent to
  \begin{equation*}
    \+U^T\+A\+V\od{(\+V^T W)}{y}
    = - \+U^T\+Q\+V (\+V^T W) + \+U^Th.
  \end{equation*}
  Since $\+Q_{1j}=\+0$, the first $p$ lines of the system 
  would give
\begin{equation}\label{eq:V2T_n}
\+G^T\+A W(y) = -\int_y^{+\infty}\!\!\+G^T h(s)\,\mathrm{d}s.
\end{equation}
Since $\mathrm{span}\{\+V_2\}=\mathrm{span}\{\+A\+G\}$, the formula
\eqref{eq:V2T_n} can uniquely determine the value of $\+V_2^T W$.

  Since $\mathrm{rank}(\+A_{21})=\mathrm{rank}(\+U_2)=r$, 
  the matrix $\+A_{21}$ is invertible and the next $r$ lines
  of the system give
  \begin{eqnarray}\label{eq:V1T_n}
	  \+V_1^TW(y) &=& -\+A_{21}^{-1}\left(\+A_{22}
	  (\+V_2^TW(y))+\+A_{23}(\+V_3^T W(y))\right)\\
	  && + \int_{y}^{+\infty}\+A_{21}^{-1}\left(
	  \+Q_{22}(\+V_2^T W(s))+ \notag
	  \+Q_{23}(\+V_3^T W(s))-\+U_2^Th(s)\right)
    \, \mathrm{d}s. 
  \end{eqnarray}
  The formula \eqref{eq:V1T_n} uniquely determines 
  $\+V_1^T W$ if $\+V_3^T W$ and $\+V_2^TW$ are known.
  
  The last $N-p-r$ equations are separated alone as 
  \begin{equation*}
    \+A_{33}\odd{y}(\+V_3^T W) = -\+Q_{33}
    (\+V_3^T W)+\+V_3^Th-\+Q_{32}\+V_2^T W
	-\+A_{32}\od{\+V_2^TW}{y},
  \end{equation*}
where $\+V_2^TW$ is given by \eqref{eq:V2T_n}.
Denote by 
\[h_3=\+Q_{33}^{-1}\left(\+V_3^Th-\+Q_{32}\+V_2^T W
	-\+A_{32}\od{\+V_2^TW}{y}\right),\]
then we have 
\[
\+Q_{33}^{-1}\+A_{33}\od{\+V_3^T W}{y} = -\+V_3^T W +h_3,
\]
where $h_3$ is a given vector relying on $h$.

From the eigenvalue decomposition \eqref{eq:T0} of
$\+Q_{33}^{-1}\+A_{33}$, we assume 
$\+T^{-1} h_3=[h_+^T,h_0^T,h_-^T]^T,$ where
the rows of $h_+,h_0$ and $h_-$ are individually
$n_+,n_0$ and $n_-$. We write 	
\begin{equation}\label{eq:V3T_n} 
	\+V_3^T W = \+T_+ z_+ + \+T_{0} z_0 + \+T_- z_-,
\end{equation}
then the characteristic equations give the solution
\begin{gather}\label{eq:z0}
	z_0 = h_0(y),\\ \label{eq:zminus}
	z_- = -\+\Lambda_-^{-1}
	\int_y^{+\infty}\!\!\exp(\+\Lambda_-^{-1}(s-y))
	h_-(s)\,\mathrm{d}s,
\end{gather}
and 
\begin{equation}\label{eq:z}
	z_+ = \exp(-\+\Lambda_+^{-1}y) z_+(0) + 
	\+\Lambda_+^{-1}\int_0^y\!\!\exp(\+\Lambda_+^{-1}
	(s-y)) h_+(s)\,\mathrm{d}s,
\end{equation}
where only $z_+(0)\in\bbR^{n_+}$ should be prescribed 
by the boundary condition.

The boundary condition \eqref{eq:B3bc} leads to 
\begin{equation}\label{eq:lin_n}
	\+B\+T_+ z_+(0) = g-\+B(\+T_{0} z_0(0)
+\+T_{-}z_-(0)).
\end{equation}
According to unique solvability of the linear
algebraic system, we can solve a unique $z_+(0)$
from \eqref{eq:lin_n} if and only if the first 
condition in Theorem \ref{thm:01} holds.

In a word, when the system is given, the variable
$\+V_1^T W$ is determined. Then $\+V_3^T W$ is
uniquely solvable if and only if the first condition
in this lemma holds. Finally, $\+V_2^T W$ can be determined. 
\end{proof}

To estimate the formal solutions \eqref{eq:V1T_n},
\eqref{eq:V2T_n} and \eqref{eq:V3T_n}, we introduce some 
Poincare type inequalities.
For any $h$ with $\|h\|_a<+\infty$, we define 
\begin{equation}\label{eq:rh}
	r(y)=-\int_y^{+\infty}\!\! h(s)\,\mathrm{d}s.
\end{equation}
Then we have $r'(y)=h(y)$ and for any component $r_i$ of $r$, 
we have 
\begin{equation*}
	|r_i(y)|^2 \leq \int_y^{+\infty}\!\! e^{-2as}\,\mathrm{d}s
	\int_y^{+\infty}\!\! e^{2as} h^Th\,\mathrm{d}s < +\infty.
\end{equation*}

\begin{lemma}\label{lem:poin}
	There exists a constant $c_0>0$ such that the Poincare 
	inequality holds
\begin{equation}\label{eq:Poi1}
	\|r\|_a\leq c_0\|r'\|_a=c_0\|h\|_a,
\end{equation}
and the trace inequality holds
\begin{equation}\label{eq:trace}
	\|r(0)\|\leq c_0\|h\|_a.
\end{equation}
\end{lemma}

\begin{proof}
	By definition, we have
	\begin{eqnarray*}
		e^{2ay} r^T(y) r(y) &=& 
		-\int_{y}^{+\infty}\!\!
		\partial_s(e^{2as} r^T r(s))\,\mathrm{d}s \\
		&=& 
		-2a\int_{y}^{+\infty}\!\!
		e^{2as} r^T r(s)\,\mathrm{d}s 
		-2\int_{y}^{+\infty}\!\!e^{2as}
		(r')^T r(s)\,\mathrm{d}s.
	\end{eqnarray*}
	Let $y=0$ and use the Cauchy-Schwarz inequality, then we have
	\begin{eqnarray*} 
		r^T(0)r(0) + 2a\|r\|_a^2 \leq 2\|r\|_a\|h\|_a.
	\end{eqnarray*}
	Since $r^T(0)r(0)\geq 0,$ we have
	\[\|r\|_a\leq \frac{1}{a}\|h\|_a.\]
	Analogously, we can ignore the term $\|r\|_a^2$ to get 
	\[\sqrt{r^T(0) r(0)}\leq \sqrt{\frac{2}{a}}\|h\|_a.\]
	This completes the proof.
\end{proof}

For convenience, we use $a\lesssim b$ to represent that
there exists a constant $c>0$ such that $a\leq cb.$
Because all the matrices considered here are finite-dimensional,
with constant coefficients, they should have a uniform
upper bound on matrix norms.

\begin{lemma}
	\label{lem:42}
	Under the assumption of Theorem \ref{thm:01}, we suppose
	the first condition in Theorem \ref{thm:01} holds. Then 
	the estimation \eqref{eq:est} holds if and
	only if the second condition in Theorem \ref{thm:01} holds,
	i.e., $\+B\+T_0=\+0.$
\end{lemma}

\begin{proof}
	Utilizing the boundedness of the matrices and Poincare 
	inequality, from \eqref{eq:V2T_n}, we have
	\[
		\|\+V_2^T W\|_a \lesssim \|\+G^T\+AW\|_a
		\lesssim \|\+G^Th\|_a \lesssim \|h\|_a.	
	\]
	Analogously, the formula \eqref{eq:V1T_n} gives
	\[
		\|\+V_1^T W\|_a \lesssim \|\+V_2^T W\|_a+\|\+V_3^T W\|_a
		+\|h\|_a \lesssim \|h\|_a+\|\+V_3^T W\|_a.
	\]
	The formula \eqref{eq:V3T_n} gives 
	\[
		\|\+V_3^T W\|_a \lesssim \|z_0\|_a + \|z_-\|_a
		+ \|z_+\|_a,
	\]
	where $z_0$ is given by \eqref{eq:z0}, 
	\[
		\|z_0\|_a \lesssim \|h_3\|_a \lesssim \|h\|_a,	
	\]
	and from \eqref{eq:zminus}, we have
	\[
		\|z_-\|_a \lesssim \|h_3\|_a \lesssim \|h\|_a.	
	\]
	Due to the condition $a<1/\lambda_{max}$, the term 
	$\|\exp\left(-\+\Lambda_+^{-1}y\right)\|_a$ is bounded and
	we have
	\[ 
	\lim_{y\rightarrow +\infty}
	\+\Lambda_+^{-1}\int_0^y\!\!\exp(\+\Lambda_+^{-1}
	(s-y)) h_+(s)\,\mathrm{d}s = 0.
	\]
	So from \eqref{eq:z}, we have
	\[
		\|z_+\|_a \lesssim \|z_+(0)\|
		\|\exp\left(-\+\Lambda_+^{-1}y\right)\|_a + \|h\|_a
		\lesssim \|z_+(0)\|+\|h\|_a.
	\]

	Now the key point lies in the estimation of $\|z_+(0)\|$.
	We can inverse $\+B\+T_+$ and solve $z_+(0)$ as
	\begin{equation}
		\label{eq:z00}
		z_+(0)=(\+B\+T_+)^{-1}\left(g-
	\+B\+T_0z_0(0)-\+B\+T_-z_-(0)\right).
	\end{equation}	
	If the second condition in Theorem \ref{thm:01} holds,
	then we have
	\[z_+(0)=(\+B\+T_+)^{-1}\left(g
	-\+B\+T_-z_-(0)\right),\]
	and 
	\begin{equation*}
	\|z_+(0)\|\lesssim \|g\|+\|z_-(0)\|.
	\end{equation*}
	Due to the trace inequality, we can see from \eqref{eq:zminus}
	that $\|z_-(0)\|\lesssim \|h\|_a$. Thus, we finally get 
	the estimation \eqref{eq:est}.

	On the contrary, if $\+B\+T_0\neq \+0,$ we consider
	$\|z_0(0)\|$. By definition, we choose $h=\+V_3 c$ with 
	$c=c(y)\in\bbR^{n_++n_0+n_-}$ such that 
	\[c=\+Q_{33}\+T\begin{bmatrix} 0 \\ c_0 \\ 0 \end{bmatrix},
		\quad c_0\in\bbR^{n_0},\]
	then we have $z_0=h_0=c_0.$ So $\|z_0(0)\|$ can not be 
	controlled by $\|h\|_a$, i.e., for any $M_0>0$, there 
	exists $c_0$ such that 
	$\|h\|_a=\|\+V_3 c\|_a=1$ but
	$\|z_0(0)\|=\|c_0(0)\|>M_0.$ From \eqref{eq:z00}, since
	$g$ and $z_-(0)$ are bounded, we can always find $c_0$
	such that $\|h\|_a=\|\+V_3 c\|_a=1$ but
	$\|z_+(0)\|>M_0.$ So the estimation \eqref{eq:est} does
	not hold when $\+B\+T_0\neq \+0.$
	
	In a word, we complete the proof.
\end{proof}

Theorem \ref{thm:01} is the combination of Lemma \ref{lem:41}
and Lemma \ref{lem:42}. From the proof of Lemma \ref{lem:42},
we can see that $\+V_3^TW(0)$ generally can not be controlled
by $\|h\|_a$ due to the contribution of $z_0(0)$.
So is $\+V_1^TW(0)$ from the formula \eqref{eq:V1T_n}. But
from \eqref{eq:V1T_n}, we have
\[\|\+A_{21}\+V_1^TW(0)+\+A_{23}\+V_3^TW(0)\|\lesssim
\|h\|_a\]
because of the Poincare inequality.

%% file: conclu.tex
\section{Conclusions}
We extended the well-posed conditions of the linear moment
system in half-space to the non-homogeneous case. The 
stability of the non-homogeneous equations would ask for 
additional criteria compared to the homogeneous case. The 
Grad boundary condition was shown unstable and thus unreliable
in reality. Some modified boundary conditions were proved 
well-posed, while the accuracy of these conditions
should be studied further.
Thanks to the analytical expressions of solutions, we can 
solve the non-homogeneous half-space problem efficiently. 

\subsubsection*{}
\begin{Ack}
This work is financially supported by the National Key R\&D
Program of China, Project Number 2020YFA0712000. 
\end{Ack}
\begin{datava}
This theoretical work uses no external dataset.
\end{datava}